\documentclass[11pt,a4paper,twoside]{amsart}
\usepackage{amsmath,amsfonts,amsthm,amsopn,color,amssymb,enumitem}
\usepackage{palatino}
\usepackage{graphicx}
\usepackage[colorlinks=true]{hyperref}
\hypersetup{urlcolor=blue, citecolor=red, linkcolor=blue}

\usepackage{cite}

\newcommand{\e}{\varepsilon}

\newcommand{\R}{\mathbb{R}}

\newcommand{\Sd}{\mathbb{S}^2}

\newcommand{\weakto}{\rightharpoonup}

\renewcommand{\d }{\delta }

\renewcommand{\l }{\lambda}

\renewcommand{\O}{\Omega}

\renewcommand{\S}{\Sigma}

\newcommand{\N}{\mathbb{N}}

\def\bbm[#1]{\mbox{\boldmath $#1$}}
\newcommand{\beq }{\begin{equation}}
\newcommand{\eeq }{\end{equation}}

\newtheorem{theorem}{Theorem}[section]
\newtheorem{lemma}[theorem]{Lemma}

\newtheorem{definition}[theorem]{Definition}
\newtheorem{proposition}[theorem]{Proposition}
\newtheorem{remark}[theorem]{Remark}

\title[Prescribing the Gaussian curvature in a subdomain of $\Sd$]
{Prescribing the Gaussian curvature in a subdomain of $\Sd$ with
Neumann boundary condition}

\author{Rafael L\'{o}pez-Soriano and David Ruiz}
\address{Dpto. An\'{a}lisis Matem\'{a}tico, Granada, 18071 Spain.}

\thanks{The authors have been supported by the Spanish Ministry of
Science and Innovation under Grant MTM2011-26717 and by J.
Andalucia (FQM 116).}

\email{rafals@ugr.es, daruiz@ugr.es}
\date{}

\keywords{Prescribed Gaussian Curvature problem, Neumann boundary condition, Variational methods.}

\subjclass[2010]{35J20, 58J32.}

\begin{document}

\begin{abstract}
The problem of prescribing the Gaussian curvature under a conformal change of the metric leads to the equation:
$$-\Delta u +2 = 2 K(x) e^u. $$
Here we are concerned with the problem posed on a subdomain
$\Sigma \subset \Sd$ under Neumann boundary condition. By using
min-max techniques we give a new existence result that generalizes
and unifies previous work on the argument.

For sign-changing $K$, compactness of solutions is not known in full generality, and
this difficulty is bypassed via an energy comparison argument.

\end{abstract}

\maketitle

\section{Introduction}
\setcounter{equation}{0}

In this paper we study the existence of solution for the problem:

\begin{equation}\label{ecua}
\left\{\begin{array}{ll}
-\Delta u +2 = 2 K(x) e^u,  \qquad & \text{in $\Sigma$,}\\
\frac{\partial u}{\partial n} = 0, \qquad  &\text{on
$\partial\Sigma $.}
\end{array}\right.
\end{equation}

Here $\Sigma$ is a smooth subdomain of the unit sphere $\Sd$
equipped with the usual metric $g_0$, $\Delta=\Delta_{g_0}$
denotes the Laplace-Beltrami operator and $K:\overline{\S}
\rightarrow \R$ is a smooth function. This equation has a strong
geometric meaning since $K(x)$ is the Gaussian curvature of the
conformal metric $g=g_0 e^u$. The problem of prescribing the
Gaussian curvature of a compact surface under a conformal change
of the metric is a classic one, starting from the pioneer work of
Kazdan and Warner \cite{KazWar}. The case of the sphere is particularly delicate and has received the name of Nirenberg problem, see, for instance, \cite{ChYg}.

Besides this geometric motivation, this kind of equations arise
also from physical models such as the abelian Chern-Simons-Higgs
theory and the Electroweak theory, see \cite{dunne, tar, yang}.

In this paper we consider the problem on a subdomain of $\Sd$
under Neumann boundary conditions. It is important to observe that
with this boundary condition \eqref{ecua} is not invariant under
conformal transformations of the sphere, as is the Nirenberg
problem.

This problem has already been considered in \cite{ChYg, guo,
li-liu,Wang} (see also \cite{Ndiaye}, where the analogue for the
Paneitz operator is considered). In this paper we use min-max
theory to give a new existence result under very general
hypotheses.

By integrating equation \eqref{ecua} we obtain that:

\begin{equation} \label{rho}
\rho:=2|\Sigma|= 2 \int_{\S} K(x)e^u.
\end{equation}
In particular, no solution exists if $K$ is negative. From now on we will assume:

\begin{enumerate}[label=(H1), ref=(H1)]

\item \label{H1} $K(x)>0$ for some $x \in \Sigma$.
\end{enumerate}

Moreover, \eqref{rho} implies that \eqref{ecua} can be rewritten
in the form:

\begin{equation}\label{ecua2}
\left\{\begin{array}{ll}
-\Delta u +2 = \rho \frac{K e^u}{\int_{\S} Ke^u},  \qquad & \text{in $\Sigma$,}\\
\frac{\partial u}{\partial n} = 0, \qquad  &\text{on $\partial
\Sigma $.}
\end{array}\right.
\end{equation}

Problem \eqref{ecua2} is the Euler-Lagrange equation of the energy functional:
$$I_{\rho}(u)=\frac{1}{2}\int_{\Sigma} \left|\nabla u\right|^2
 +2 \int_{\Sigma} u - \rho \log \int_{\Sigma} K
e^u ,$$
defined in the domain
$$ X:=\left\{u \in H^1(\Sigma) \mid \int_{\Sigma} K e^u  >0 \right\}.$$

Observe that assumption \ref{H1} implies that $X$ is not empty.
The functional $I_\rho$ is invariant under addition of constants,
as well as problem \eqref{ecua2}.

In \cite{ChYg} it was shown that $I_\rho$ is always bounded from
below and coercive if $\rho <4 \pi$ (that is, $|\Sigma| < 2\pi$).
Therefore, a solution is obtained by minimization. The case $\rho=
4\pi$ is critical, $I_\rho$ is still bounded from below but loses
coercivity. Moreover, the problem may present loss of compactness
due to bubbling of solutions. Some results in this direction have
been given in \cite{ChYg, li-liu}.

In this paper we consider the case $\rho \in (4\pi, 8\pi)$, that
is, $|\Sigma|>2\pi$. The case:

\begin{enumerate}[label=(Q1), ref=(Q1)]

\item \label{Q1} $K(x)<0$ for any $x \in \partial \Sigma$,
\end{enumerate}
was already considered in \cite{guo}. Under \ref{Q1} $I_\rho$ is still bounded from below and coercive, and a
solution can be found by minimization.

Instead, if $K(x)>0$ on some point $x \in \partial \S$, then $I_\rho$ is no longer bounded from below. In order to find critical points of saddle type, min-max arguments appear as the natural technique to handle the problem. A first result in this direction was given in \cite{Wang}, where the existence of a solution for \eqref{ecua} is shown under the assumption:

\begin{enumerate}[label=(Q2), ref=(Q2)]

\item \label{Q2} $\partial \Sigma$ is disconnected and $K(x)>0$ for any  $x  \in \partial \Sigma$.
\end{enumerate}

The main goal of this paper is to extend the existence results of \cite{guo}
and \cite{Wang} under a unique general condition, namely:

\begin{enumerate}[label=(H2), ref=(H2)]
\item \label{H2} $K(x)\neq 0$ for any $x \in \partial \Sigma$.
\end{enumerate}

\begin{theorem}\label{main}
Let $\Sigma$ be a smooth domain of $\Sd$ with $|\Sigma|>2\pi$, and
$K:\overline{\Sigma} \to \R$ a regular function satisfying \ref{H1} and \ref{H2}. Then, problem \eqref{ecua} admits a solution.
\end{theorem}

Let us emphasize that our assumption \ref{H2} contains both \ref{Q1} and \ref{Q2} as particular cases. Moreover, our proofs fix
some details in the proof of \cite{Wang}, as will be
explained below.

The proof is based on a detailed study of the low energy levels of
the functional $I_\rho$, in the spirit of \cite{djlw, djadli, dm,
mr}. Take $p \in \partial \Sigma$ such that $K(p)>0$, and define,
for any $\l \in \R$,

\begin{equation}\label{bubble}
\psi_{\lambda}(p)(x)=2\log\left(\frac{\lambda^2}{1+\lambda^2dist^2(x,p)}\right).
\end{equation}
It is easy to check that, for $\rho > 4\pi$,
$$I_{\rho}(\psi_{\lambda}) \rightarrow - \infty, \quad \mbox{ as } \l \to +\infty.$$

The functions $\psi_\l$ are usually called "bubbles", and $e^{\psi_\l}$ is
the conformal factor of the stereographic map, composed with a
translation and a dilation. It is important to observe that
$$\frac{e^{\psi_\l}}{\int_\Sigma e^{\psi_\l}}
\weakto \delta_p, \quad \mbox{ as } \l \to +\infty,$$
in the sense of weak convergence of measures.

By using Chen-Li type inequalities (see \cite{ChenLi}), we shall
prove that any $u_n \in X$ such that $I_\rho(u_n) \to -\infty$
satisfies this kind of concentration, namely,
$$\frac{e^{u_n}}{\int_\Sigma e^{u_n}}
\weakto \delta_q, \mbox{ for some } q \in \partial \Sigma,\
K(q)>0.$$

Let us denote:
$$ \partial \Sigma = \bigcup_{i=1}^l \O_i, $$
where $l \in \N$ and $\O_i$ are the connected components of
$\partial \Sigma$, and
$$ \O^+:=\bigcup \left\{ \O_i: K(x)>0 \mbox{ in } \O_i \right\}.$$

Roughly speaking, functions $u$ with low values of $I_\rho(u)$
concentrate around a point of $\O^+$. The nontrivial topology of
$\O^+$ implies that the low energy levels of $I_\rho$ are not
contractible, and this allows us to use a min-max argument.

The classic Palais-Smale property is not known to hold for this
problem, and this problem is bypassed by using the so-called
monotonicity trick of Struwe \cite{struwe}. For that, compactness
of solutions of \eqref{ecua} is required, which is typically proved
by a quantization argument as in \cite{breme,lisha}. Such
a quantization result was claimed in \cite{Wang}, but the proof
there works only for strictly positive $K$. For $K$ vanishing at a
point, for instance, the quantization result is altered, as shown
by Bartolucci-Tarantello \cite{btjde02, btcmp02}.

Compactness of solutions for \eqref{ecua} with sign-changing $K$
is not known, up to our knowledge. The only related compactness
result we are aware of is \cite{cl-dcds}.
There the Nirenberg problem is addressed; it is not clear to us whether those
arguments can be translated to our setting or not. Moreover, \cite{cl-dcds} needs further assumptions involving
the zero set of $K$.

In this paper we are able to deduce compactness of our solutions
by using an energy comparison argument. This argument seems to be
completely new in this kind of problems. Indeed, the energy
estimates on our solutions already imply a boundary concentration
of a possibly divergent sequence of solutions $u_n$, see
Proposition \ref{?}. With that first step, the classic arguments
of Li-Shafrir \cite{lisha} apply.

The rest of the paper is organized as follows. In Section 2 we
fix the notation and give some preliminary results, like
Chen-Li type inequalities and Proposition \ref{?}. Section 3 is
devoted to the proof of Theorem \ref{main}.

\section{Notations and preliminary results}
\setcounter{equation}{0}

In this section we establish the notation to be used throughout the rest of the paper and collect some preliminary facts.

We denote by $dist(x,y)$ the distance between two points $x,y \in \S$. $B(p,r)$ will be the metric ball of radius $r$ and center $p$. Let $u \in L^1(\S)$, $\overline{u}=\frac{1}{\left|\S\right|} \int_{\S} u$ stands for the average of $u$, where $\left|\S\right|$ is the area of $\S$. For $L$ a fixed real number, we denote the sublevel of $I_{\rho}$ by the set $I^{L}_{\rho}:=\{u\in X : I_{\rho}(u) < L \}$. Let $\delta$ be a positive parameter, we define $A^{\delta}$ as the exterior parallel set of $A \subset \S$, i.e., $A^{\delta}= \left\{x \in \S \mid dist(x,A)<\delta\right\}$.

\

A powerful tool in our study is Moser-Trudinger type inequalities, which will allow us to deduce properties about $I_{\rho}$. In this direction, we start by recalling some of their versions.

\begin{proposition}
There exists $C>0$ such that

\begin{equation}\label{eMTd}
\log \int_{\Sigma} e^{u} \leq \frac{1}{16\pi} \int_{\Sigma} \left|\nabla u\right|^2 + C, \quad \forall  u \in H^1_0(\Sigma),
\end{equation}
and

\begin{equation}\label{eMT1d}
\log \int_{\Sigma} e^{u} \leq \frac{1}{8\pi} \int_{\Sigma} \left|\nabla u\right|^2 + C, \quad \forall  u \in H^1(\Sigma) \text{ with} \int_{\S} u =0.
\end{equation}
\end{proposition}

The first inequality is a weaker version of the classic one of Moser-Trudinger inequality, see \cite{Mos1,Mos2,Tru1}. The second one is a corollary of Proposition 2.3 in \cite{ChYg}, which was conceived for plane domains and for domains of $\Sd$.

The constants multiplying the Dirichlet energy are optimal. In other words, for $\alpha$ less than $\frac 1 {8\pi}$ and $\frac 1 {16\pi}$ respectively, using the standard bubbles \eqref{bubble} peaked at some point of $\S$, one can check that \eqref{eMTd} and \eqref{eMT1d} do not hold.

The constant in \eqref{eMT1d} appears multiplied by two in relation to \eqref{eMTd}, since we can center a bubble on a point of $\partial \S$, with which its volume and Dirichlet energy is divided approximatively by two. However, this process can not occur in \eqref{eMTd} because of its boundary condition.

As an easy application of the previous proposition, we have
$$ I_{\rho}(u) \geq \frac{4\pi -\rho}{8\pi}\Vert u \Vert^{2}_{H^1(\Sigma)} + C,$$
for all $u \in X$. In particular, $I_{\rho}$ is coercive for $\rho \in \left(0, 4\pi \right)$, and a solution for \eqref{ecua} can be found as a minimizer.

The following localized version of Moser-Trudinger type inequalities will be of use.

\begin{proposition}\label{MTl}
Let $\e > 0$, $\delta >  0$ and $\Sigma_1 \subset \Sigma$ such that $dist(\Sigma_1,\partial\Sigma)>\delta$.
Then, there exists a constant $C=C(\e,\delta)$, such that for every $u \in H^1(\Sigma)$ with $\int_{\S} u =0$,
$$16\pi \log \int_{\Sigma_1} e^{u} \leq \int_{\Sigma^{\delta}_1} \left|\nabla u\right|^2 + \e\int_{\Sigma} \left|\nabla u\right|^2 + C.$$
\end{proposition}

\begin{proof}

Our argument follows closely the proof used in Theorem 2.1 in \cite{ChenLi} (see also \cite{mr}).

First, we consider a smooth cutoff function $g$ with values into $\left[0,1\right]$ satisfying

\begin{equation}\label{cutoff}
\left\{\begin{array}{ll}
g = 1,  & \forall x \in \Sigma_1, \\
g = 0, & \forall x \in \Sigma \setminus \Sigma^{\delta / 2}_1. \\
\end{array}\right.
\end{equation}
Clearly, function $gu \in H^1_0(\S)$. Applying inequality \eqref{eMTd} to $gu$ we obtain
$$ 16\pi \log \int_{\Sigma_1} e^{u} \leq 16\pi \log \int_{\Sigma} e^{gu} \leq \int_{\Sigma} \left|\nabla (gu)\right|^2 + C.$$
Using the Leibnitz rule to the gradient we have

\begin{equation}\label{MTl0}
\int_{\Sigma} \left|\nabla (gu)\right|^2 \leq \int_{\Sigma^{\delta}_1} \left|\nabla u\right|^2 + 2   \int_{\Sigma} gu \nabla g \nabla u +  C(\delta) \int_{\Sigma} u^2.
\end{equation}
By Cauchy's inequality,

\begin{equation}\label{MTl1}
\int_{\Sigma} gu \nabla g \nabla u \leq \e \int_{\Sigma} \left| \nabla u \right|^2 + C(\e,\delta)\int_{\Sigma} u^2.
\end{equation}
Combining \eqref{MTl0} and \eqref{MTl1}, we get

\begin{equation}\label{MTl2}
16\pi \log \int_{\Sigma_1} e^{u} \leq \ \int_{\Sigma^{\delta}_1} \left|\nabla u \right|^2 + \e  \int_{\Sigma} \left|\nabla u \right|^2 + C(\e,\delta) \int_{\Sigma} u^2.
\end{equation}

Let us now estimate the last term of \eqref{MTl2}. Take $\eta$ such that $\left|\{ x \in \Sigma : u(x) \geq a \} \right| = \eta$. Let $\left(u-a \right)^+ = max \{ 0, u-a \}$ and applying \eqref{MTl2}, we obtain

\begin{eqnarray}\label{MTl3}
& \displaystyle{16\pi \log \int_{\Sigma_1} e^u \leq 16\pi \log \left\{ e^a \int_{\Sigma_1} e^{(u-a)^{+}} \right\}} \nonumber \\
& \displaystyle{\leq 16\pi a + \int_{\Sigma^{\delta}_1} \left|\nabla u \right|^2 + \e  \int_{\Sigma} \left|\nabla u \right|^2 + C(\e,\delta) \int_{\Sigma} \left( (u-a)^+ \right)^2}.
\end{eqnarray}
By H\"older and Sobolev inequalities

\begin{equation}\label{MTl4}
\int_{\Sigma} \left( (u-a)^+ \right)^2 \leq {\eta}^{1/2} \left(\int_{\Sigma} \left( (u-a)^+ \right)^4 \right)^{1/2} \leq c {\eta}^{1/2} \int_{\Sigma} \left| \nabla u \right|^2,
\end{equation}
and by Poincar\'e inequality

\begin{equation}\label{MTl5}
a \eta \leq \int_{a \leq u} u \leq \int_{\Sigma} \left|u \right| \leq c \left( \int_{\Sigma} \left| \nabla u \right|^2\right)^{1/2}.
\end{equation}
Hence for every $\delta >0$, from \eqref{MTl5} by Cauchy's inequality,

\begin{equation}\label{MTl6}
a \leq \delta\int_{\Sigma} \left| \nabla u \right|^2 + \frac{c^2}{4\delta{\eta}^2}.
\end{equation}
Finally, let $\eta$ satisfying

\begin{equation}\label{MTl7}
\eta^{1/2} \geq \frac{C(\e,\delta)}{\e}.
\end{equation}
From \eqref{MTl3}, \eqref{MTl4}, \eqref{MTl6} and \eqref{MTl7}, we conclude the proof.

\end{proof}

\begin{proposition}\label{ChYgl}
Let $\e > 0$, $\delta > 0$ and $\Sigma_1 \subset \Sigma$. Then, there exists a constant $C=C(\e,\delta)$, such that for every $u \in H^1(\Sigma)$ with $\int_{\S} u =0$,
$$ 8\pi \log \int_{\Sigma_1} e^{u} \leq \int_{\Sigma^{\delta}_1} \left|\nabla u\right|^2 + \e\int_{\Sigma} \left|\nabla u\right|^2 + C.$$
\end{proposition}
\begin{proof}

By an analogous argument to the one used in Proposition \ref{MTl}, we prove the result. Let $g$ as we defined in \eqref{cutoff} and applying \eqref{eMT1d} to $gu$, we obtain
$$ 8\pi \log \int_{\Sigma_1} e^{u} \leq 8\pi \log \int_{\Sigma} e^{gu} \leq \int_{\Sigma} \left|\nabla (gu)\right|^2 + \overline{gu} + C.$$
Now, we estimate the average of $gu$ as
$$ \overline{gu} \leq C(\delta) + c\int_{\Sigma} u^2.$$
Then, as we did in \eqref{MTl2}, we have
$$ 8\pi \log \int_{\Sigma_1} e^{u} \leq \int_{\Sigma^{\delta}_1} \left|\nabla u \right|^2 + \e  \int_{\Sigma} \left|\nabla u \right|^2 + C\int_{\Sigma} u^2 + C.$$

It suffices to estimate $\int_{\Sigma} u^2$ exactly in the same way from the previous proof.

\end{proof}

Observe the difference between the choice of $\S_1$ in both propositions. Whereas in the first result $\S_1$ is away from the boundary of $\S$, there is no restriction in that sense in the second one.

As a consequence of the last results, we present a version of the Chen-Li inequality \cite{ChenLi} (see also \cite{Aub}). This version was first stated in \cite{Wang}. Roughly speaking, it states that if $e^u$ is \textit{spread} into two regions of $\S$, then $I_\rho$ is still bounded from below. And the same occurs if $e^u$ has \textit{mass inside} $\S$.

\begin{lemma}\label{ChLi}
Let $\e>0$, $\delta>0$ and $0<\gamma<1/2$. Let $\Sigma_1$, $\Sigma_2$ and $S$ be subsets of $\Sigma$, such that $\Sigma^{\delta}_1 \bigcap \Sigma^{\delta}_2 = \emptyset$ and $S^{\delta} \bigcap \partial \Sigma = \emptyset$. If

\begin{equation}\label{ChLi1}
\frac{ \int_{\Sigma_1} e^u}{\int_{\Sigma} e^u } \geq \gamma, \hspace{0.5cm} \frac{ \int_{\Sigma_2} e^u}{\int_{\Sigma} e^u } \geq \gamma,
\end{equation}
or

\begin{equation}\label{ChLi2}
\frac{ \int_{S} e^u}{\int_{\Sigma} e^u } \geq \gamma,
\end{equation}
then, there exists a constant $C=C(\e,\delta,\gamma)$, such that for all $u \in H^1(\Sigma)$ satisfying $\int_{\Sigma} u =0$,
$$(16\pi-\e) \log \int_{\Sigma} e^u \leq  \hspace{0.1cm} \int_{\Sigma} \left|\nabla u\right|^2 + C.$$
\end{lemma}

\begin{proof}
First, we take $\Sigma_1, \Sigma_2$ verifying \eqref{ChLi1} and apply Proposition \ref{ChYgl} for each one
$$ 8\pi \log \int_{\Sigma} e^{u} \leq 8\pi \log \left( \frac{1}{\gamma} \int_{\Sigma_i} e^{u} \right) \leq \int_{\Sigma^{\delta}_{i}} \left|\nabla u\right|^2 + \e\int_{\Sigma} \left|\nabla u\right|^2 + C(\e,\delta,\gamma), \, i=1,2.$$

Finally, we add both expressions
$$ 16\pi \log \int_{\Sigma} e^{u} \leq \int_{\Sigma^{\delta}_{1} \bigcup \Sigma^{\delta}_{2} } \left|\nabla u\right|^2 + 2 \e\int_{\Sigma} \left|\nabla u\right|^2 + C(\e,\delta,\gamma) \leq (1+2\e) \int_{\Sigma } \left|\nabla u\right|^2 + C(\e,\delta,\gamma),$$
concluding the first case.

For $S$ satisfying \eqref{ChLi2}, we can come to the same conclusion by using Proposition \ref{MTl}.

\end{proof}

The following proposition will be crucial in Section 3 not only for the min-max argument, but also for the compactness result (see Proposition \ref{ultima}).

\begin{proposition}\label{?}
Let $\rho$ a fixed constant in $\left( 4\pi, 8\pi \right)$, $\{u_n\}$ a sequence in $X$ such that $I_{\rho}(u_n)<C$. Suppose $\O^+ \neq \emptyset$. If

\begin{equation}\label{?h}
\Vert u_n \Vert_{H^1(\S)} \rightarrow \infty,
\end{equation}
then, up to a subsequence,
$$\frac{e^{u_n}}{\int_{\Sigma} e^{u_n}} \weakto \delta_{p} \quad \text{with $p \in \O^+$}.$$
\end{proposition}

\begin{proof}
Let $C>0$ be such that $K(x)<C$ for all $x \in \Sigma$. We define
$$ E_{\rho}(u)=\frac{1}{2}\int_{\Sigma} \left|\nabla u\right|^2  + 2 \int_{\Sigma} u - \rho \log \int_{\Sigma} C e^u, \quad \text{in $H^1(\Sigma)$}.$$
Easily, we can show
$$I_{\rho}(u_n) > E_{\rho}(u_n),$$
for all $u_n \in X$, hence $E_{\rho}(u_n)$ is also bounded from above.
We introduce the measures
$$\mu_n = \frac{e^{u_n}}{\int_{\Sigma}e^{u_n}},$$
which satisfies that $\mu_n \rightharpoonup \mu$ in the sense of weak convergence of measures, up to subsequence.

Let $\tau>0$ and the open subset $M_{\tau}=\left\{ x \in \Sigma : dist(x,\partial \Sigma)>\tau \right\}$. Suppose $\mu (M_{\tau}) > 0$. By the weak convergence of measures, we have
$$ 0 < \mu(M_{\tau}) \leq \displaystyle\liminf_{n \rightarrow \infty} \mu_n(M_{\tau}).$$

From this inequality, since \eqref{ChLi2} holds, we can use Lemma \ref{ChLi} for any $\e>0$ to obtain

\begin{equation}\label{pobj2}
E_{\rho}(u_n) > \left( \frac{1}{2} - \frac{\rho}{16\pi - \e} \right) \Vert{u_n}\Vert^2_{H^1(\S)} + C'.
\end{equation}
By \eqref{?h}, \eqref{pobj2} contradicts the above boundedness of $E_{\rho}(u_n)$. Therefore $\mu\left(M_{\tau}\right)=0$  for every $\tau>0$, i.e., $\mu$ is supported in $\partial \S$.

Now, for any $r>0$ define $\hat{B}(p,r)=B(p,r) \bigcap \Sigma$ with $p \in \overline{\Sigma}$. Take a covering of $\partial \S$ with a finite number of $\hat{B}(p_i,r)$, where $p_i \in \partial \S$. Clearly, there exists $p_j \in \partial \S$ such that $\mu(\hat{B}(p_j,r))>0$.

Observe that if $\mu(\partial \Sigma \setminus \hat{B}(p_j,2r))>0$, we conclude again \eqref{pobj2} by Lemma \ref{ChLi}. Therefore, $\mu(\hat{B}(p_j,2r))=1$.

Taking $r_n=\frac 1 n$, there exists a sequence $\{ p_n\} \subset \partial \S$ such that $\mu(\hat{B}(p_n,2r_n))=1$. Passing to a subsequence, $p_n \rightarrow p$ as $r_n \rightarrow 0$. So,
$$ \mu = \delta_{p},$$
as claimed. Moreover, since $u_n \in X$,
$$ 0<\frac{\int_{\S}K e^{u_n}}{\int_{\Sigma} e^{u_n}} \rightarrow K(p),$$
which implies that $p \in \Omega^+$.

\end{proof}

\section{Proof of Theorem \ref{main}}
\setcounter{equation}{0}

In this section we take profit of the previous results to prove Theorem \ref{main}. The proof uses a variational argument, together with a compactness result on the solutions of \eqref{ecua} with bounded energy. The variational scheme uses in
an essential way Proposition \ref{?}, in the spirit of \cite{djadli, dm, mr, Wang}.

First of all, let us observe that Proposition \ref{ChLi} yields easily the existence of a solution if $K$ is negative in $\partial \S$. Indeed, in such case $\Omega^+ = \emptyset$, and Proposition \ref{ChLi} implies that $I_\rho$ is coercive. It is well-known that it is also weak lower semicontinuous, and therefore it attains its infimum. This case, however, was already treated in \cite{guo}, and is not the purpose of this paper.

Therefore, we will assume in what follows that $K$ is positive on some connected component of $\Sigma$, so that $\Omega^+ \neq \emptyset$. In such case, we are able to construct functions with arbitrary low energy level, as follows:

\begin{lemma}\label{lema1}
For any $\lambda>0$ and $p \in \Omega^+$, let us define:
$$ \psi_{\lambda}(p): \Sigma \to \R, \ \psi_{\lambda}(p)(x)=2\log\left(\frac{\lambda^2}{1+\lambda^2dist^2(x,p)}\right).$$
Then, for any $L>0$, there exists $\l(L)$ such that for any $\l \geq \l(L)$, $p \in \Omega^+$, $I_\rho(\psi_{\lambda}(p))<-L$.

\end{lemma}

\begin{proof}

For $\tau>0$ fixed and $B(p,\tau) \subset \Sd$, the following asymptotic estimate is known (see e.g. Section 4 in \cite{Ma})
$$ \int_{B(p,\tau)} e^{\psi_{\lambda}(p)(x)}  = \pi \lambda^2 + o(\lambda^2).$$

Taking into account the smoothness of $\Sigma$, we obtain
$$ \int_{\Sigma} K(x)e^{\psi_{\lambda}(p)(x)}
 = \frac{K(p)}{2} \pi \lambda^2 + o(\lambda^2).$$

It is also well-known
$$ \int_{B(p,\tau)} |\nabla \psi_{\lambda}(p)(x)|^2 \,  = 32 \pi \log \lambda + o(\log \l),$$
and hence,
$$ \int_\Sigma |\nabla \psi_{\lambda}(p)(x)|^2 \,  = 16 \pi \log \lambda + o(\log \l).$$

Furthermore,
$$ \int_{\Sigma} \psi_{\lambda}(p)(x)  =O(1).$$

Finally, an easy compactness argument implies that the error terms in the above estimates are independent of the point $p \in \O^+$, concluding the proof of the lemma.

\end{proof}

The above lemma implies, in particular, that $I_\rho$ is unbounded from below. But it gives much more information: indeed, given any $L>0$, we can choose  $\l$ so that the following continuous map is well-defined:

\begin{align*}
\Psi_{\lambda}: \Omega^+ &\rightarrow I_{\rho}^{-L} \\
p &\mapsto \psi_{\lambda}(p).
\end{align*}

Observe that those functions concentrate, as $\l\to +\infty$, around $p\in \O^+$. Now we plan to show that, indeed, any function $u$ in a low sublevel of $I_\rho$ must behave in that fashion. This idea is made explicit by a reverse map, that is, a continuous map $\Phi: I_\rho^{-L} \to \Omega^+$, for $L$ large. This map, together with $\Psi$, will give us useful information about the topology of low energy sub-levels of $I_\rho$.

First, let us introduce the center of mass of the function $e^u$, defined as:

\begin{equation} \label{center}
P(u):=\frac{\int_{\Sigma} xe^u}{\int_{\Sigma} e^u} \in \mathbb{R}^3.
\end{equation}

In our next result we show that low sub-levels have center of mass in an arbitrary small neighborhood of $\O^+$ in $\R^3$, denoted by:
$$ N(\Omega^+, \delta)= \{x \in \R^3:\ d(x,\Omega^+)<\delta\},$$
where $d$ refers to the Euclidean distance in $\R^3$.

\begin{proposition}\label{bla}
Given any $\delta>0$, there exists $L(\delta)>0$ such that for any $L>L(\delta)$, we have that
$
P\left(I_{\rho}^{-L} \right) \subset N(\Omega^+, \delta).
$
\end{proposition}

\begin{proof}

Take $u_n \in X$ with $I_\rho(u_n) \to -\infty$. Obviously, it must be an unbounded sequence. By Proposition \ref{?},
$$ \frac{ e^{u_n}}{\int_{\Sigma} e^{u_n}} \weakto \d_p, \ p \in \Omega^+ \Rightarrow
P(u_n)=\frac{\int_{\Sigma} xe^{u_n}}{\int_{\Sigma} e^{u_n}} \to p.$$

\end{proof}

\begin{definition} \label{L0} Because of the smoothness of $\Sigma$, there exists $\d_0 >0$ and a continuous retraction
$$\Pi: N(\Omega^+, \delta_0) \rightarrow \Omega^+.$$

Therefore, there exists $L_0=L(\delta_0)$ as in Proposition \ref{bla} such that for any $L>L_0$, we can define the reverse map $$\Phi= \Pi \circ P:I_\rho^{-L} \to \O^+.$$
\end{definition}

Next proposition will be the key point for our min-max argument.

\begin{proposition}\label{H} Fix any $L>L_0$ and take $\l>\l(L)$ where $\l(L)$ is given in Lemma \ref{lema1}. Then the composition $\Phi \circ \Psi_{\lambda}: \O^+ \rightarrow \O^+$ is homotopically equivalent to the identity map. Moreover, $\Psi_\lambda(\Omega^+)$ is not contractible in $I_\rho^{-L}$.
\end{proposition}

\begin{proof}

Let us define the homotopy \begin{align*}
H:&\left[0,1\right] \times \Omega^+\ \rightarrow \Omega^+\\
& (t,p) \mapsto H(t,p):= \Phi \circ \Psi_{\lambda(t)}(p),
\end{align*}
where $\lambda(0)=\l$ and $\lambda(t)$ is an increasing continuous function with  $\lambda(t) \rightarrow \infty$ as $t \rightarrow 1$.

Let us show first that $H(t,\cdot) \to Id|_{\Omega^+}$ as $t\to 1$. Take $p_n \to p \in \Omega^+$, $\lambda_n \to +\infty$; by the proof of Lemma \ref{lema1},
$$ \frac{e^{\psi_{\lambda_n}(p_n)}}{\int_{\Sigma} e^{\psi_{\lambda_n}(p_n)}}  \weakto \delta_p.$$

As a consequence,
$$ P \circ \Psi_{\lambda_n}(p_n) \rightarrow p.$$

The second assertion of Proposition \ref{H} follows easily from the former and the fact that $\Omega^+$ is a non-contractible set.

\end{proof}

Take any $v \in X$ fixed, and define:
$$ \mathcal{C}:=\left\{ \log\{ t\exp\{\psi_{\lambda}(p)\}+ (1-t)\exp\{v\} \}: p \in \Omega^+,  t \in \left[0,1 \right]\right\}.$$

It is easy to check that $\mathcal{C}$ is contained in $X$. As topology is concerned, $\mathcal{C}$ is a cone with base $\Psi_\l(\Omega^+) \sim \O^+$, so that $\partial \mathcal{C}= \Psi_\l(\Omega^+)$. In other words, $\mathcal{C}$ is the union of a finite number of circular cones, each of them containing a connected component of $\Psi_\l(\Omega^+)$ in its base, such that their vertices coincide at $v$.

We now define the min-max value of $I_\rho$ on suitable deformations of $\mathcal{C}$, namely:
\begin{definition}\label{minmax}
$$
\alpha_{\rho}:=\displaystyle \inf_{\eta\in
\Gamma} \max_{u\in \mathcal{C}} I_{\rho}(\eta(u)),
$$
with
$$
\Gamma:=\left\{\eta:\mathcal{C}\rightarrow X \ \mbox{continuous}:\ \eta(u)=u \ \forall \ u \in  \Psi_\lambda(\Omega^+) \right\}.
$$
\end{definition}

\begin{lemma}\label{lemma minmax}
$\alpha_{\rho} \geq -L_0$, where $L_0$ is given in Definition \ref{L0}.
\end{lemma}

\begin{proof}

Take $L>L_0$; for any deformation $\eta \in \Gamma$,  $\partial
\mathcal{C} = \Psi_\l(\Omega^+)$ is contractible in
$\eta(\mathcal{C})$. Moreover, Proposition \ref{H} establishes
that $\Psi_\l(\Omega^+)$ is not contractible in $I_{\rho}^{-L}$.
Therefore, $\eta(\mathcal{C}) \nsubseteq I_{\rho}^{-L}$, that is,
there exists $\hat{u} \in \mathcal{C}$ with
$I_{\rho}(\eta(\hat{u})) \geq -L$. This concludes the proof.

\end{proof}

Therefore, take $L>L_0$ and $\l>\l(L)$ where $\l(L)$ is given in
Lemma \ref{lema1}. Lemma \ref{lemma minmax} implies that
$\alpha_\rho > \max\{I_\rho(u): \ u \in
\Psi_{\lambda}(\Omega^+)\}$, which provides us with a min-max structure. Therefore, we can conclude the existence of a Palais-Smale sequence at level $\alpha_\rho$. However, the
boundedness of Palais-Smale sequences is still unknown for this
kind of problems. The derivation of a solution follows an argument
first used by Struwe, \cite{struwe}. This argument has been used
many times in this and other types of problems, see \cite{djlw, djadli, dm}, so we will be sketchy. An essential ingredient is the following lemma:

\begin{lemma}\label{monomini}
The function $\rho \mapsto \frac{\alpha_{\rho}}{\rho}$ is monotonically decreasing.
\end{lemma}

\begin{proof}
Just observe that, for $\rho < \rho'$,
$$ \frac{I_{\rho}(u)}{\rho}-\frac{I_{\rho'}(u)}{\rho'} = \frac{1}{2} \left( \frac{1}{\rho} - \frac{1}{\rho'} \right) \int_{\Sigma} \left| \nabla u \right|^2 \geq 0. $$

Since $\alpha_\rho$ is a min-max value for $I\rho$, the previous estimate implies the monotonicity of $\frac{\alpha_{\rho}}{\rho}$.

\end{proof}

In this setting, we obtain the following:

\begin{proposition}\label{subdens}
There exists a set $E \subset (4\pi, 8\pi)$ such that:
\begin{enumerate}
\item $(4\pi, 8 \pi)\setminus E$ has zero Lebesgue measure, and
\item for any $\rho \in E$ there exists a solution $u_\rho$ of \eqref{ecua} with $I_\rho(u_\rho)=\alpha_\rho$.

\end{enumerate}

\end{proposition}

\begin{proof}

Define
$$ E:=\left\{ \rho \in  (4\pi, 8\pi):\ \mbox{ the map } \rho \mapsto \alpha_\rho \mbox{ is differentiable at} \ \rho\right\}.$$

By Lemma \ref{monomini}, $(4\pi, 8\pi) \setminus E$ has zero measure. Fixed $\rho \in E$, take $\e>0$ sufficiently small.
Observe that the above min-max scheme is valid for values of the parameter in the interval $(\rho -\e, \rho+\e)$. In this situation, it is well-known that there exists a sequence $u_n$ satisfying:

\begin{enumerate}
\item $u_n$ is bounded in $H^1(\S)$,
\item $I_\rho(u_n) \to \alpha_\rho$,
\item $I_\rho'(u_n) \to 0$.
\end{enumerate}

That is, for almost all values of $\rho$ we can assure the
existence of a \emph{bounded} (PS) sequence. This kind of argument
was first developed in \cite{struwe} (see also \cite{djlw, djadli,
dm}).

Since $u_n$ is bounded, up to a subsequence, $u_n \weakto u_\rho$. Standard arguments show then that actually $u_n \to u_\rho$ strongly and that $u_\rho$ is a critical point for $I_\rho$.

\end{proof}

So far, we have proved the existence of a solution for $\eqref{ecua}$ for almost all values of $\rho \in (4\pi, 8\pi)$. Now, our intention is to extend this existence result for any $\rho \in (4\pi, 8\pi)$.

Typically, this is accomplished by a compactness argument via a quantization result, in the spirit of Brezis-Merle and Li-Shafrir \cite{breme, lisha}. However, the fact that $K$ may change sign is
a serious obstacle for this quantization. In \cite{Wang} it is claimed that if $u_n$ is an unbounded sequence of solutions of \eqref{ecua} with $\rho=\rho_n$, then
$$ \rho_n \to 4 k \pi, k \in \N.$$

However the derivation of this result in \cite{Wang} is correct
only for strictly positive $K$. Indeed, even for $K$ vanishing at
a point, other limit values can be achieved, as shown in
\cite{btjde02, btcmp02}. Observe, moreover, that in our setting no assumption is made on the set of zeroes of $K$, apart of being disjoint with $\partial \Sigma$.

Here we bypass this problem by noting that the solutions given by Proposition \ref{subdens} have bounded energy. This energy control, together with Proposition \ref{?}, implies already a certain concentration behavior of the sequence of solutions, if unbounded. Since $K$ is strictly positive on $\O^+$, the argument of \cite{lisha} yields the desired contradiction.

\begin{proposition} \label{ultima} Let $\rho_n,\ \rho_0 \in (4\pi, 8\pi)$, $\rho_n \to \rho_0$, and $u_n$ solutions of \eqref{ecua} for $\rho= \rho_n$. Assume also that $I_{\rho_n}(u_n)$ is bounded from above. Then, up to a subsequence, $u_n \to u_0$, and $u_0$ is a solution of \eqref{ecua} for $\rho= \rho_0$.

\end{proposition}

\begin{proof}
If $u_n$ is bounded, up to a subsequence, $u_n \weakto u_0$. Standard elliptic arguments show that the convergence is strong and that $u_0$ is the required solution.

Assume now that $u_n$ is unbounded. By Proposition \ref{?}, there exists $p \in \O^+$ with
$$ \frac{ e^{u_n}}{\int_{\Sigma} e^{u_n}} \weakto \d_p.$$

Clearly,
$$\frac{\int_{\Sigma} K(x) e^{u_n}}{\int_{\Sigma} e^{u_n}}\rightarrow K(p)>0.$$

Take $\tau>0$ so that $K(x)>0$ in $B(p,\tau)\cap \Sigma$. First, observe that

\begin{equation}\label{hola}  \frac{\int_{\Sigma \setminus B(p,\tau)} K(x) e^{u_n}}{\int_{\Sigma} K(x) e^{u_n}}  \to 0.\end{equation}

Moreover, by the quantization result of Li-Shafrir in \cite{lisha} (see also \cite{WangWei}, where the Neumann boundary condition case is treated), we obtain that:

\begin{equation} \label{hola2} \rho_n \frac{\int_{B(p,\tau)\cap \Sigma} K(x) e^{u_n}}{\int_{\Sigma} K(x) e^{u_n}} \to 4 k \pi, \ k \in \N.\end{equation}

Equations \eqref{hola}, \eqref{hola2} imply that $\rho_n \to 4 k \pi$ with $k\in \N$, a contradiction.

\end{proof}

We can now finish the proof of Theorem \ref{main}. Take any $\rho_0 \in (4\pi, 8\pi)$ and $\rho_n \in E$, $\rho_n \to \rho$.
Let $u_n$ denote the solutions of \eqref{ecua} for $\rho=\rho_n$
given by Proposition \ref{subdens}. Recall that
$I_{\rho_n}(u_n)=\alpha_{\rho_n}$, which is bounded (for instance,
by Lemma \ref{monomini}). Proposition \ref{ultima} allows us to
conclude.

\begin{remark} The argument of the proof works perfectly well if $\Sigma$ is a subdomain of any compact surface $\Lambda$, and $g_0$ is any Riemannian metric on $\Lambda$. In this general case, though, equation \eqref{ecua} loses its geometrical interpretation.

Observe that we can assume that $\Lambda$ is isometrically embedded in $\R^k$; therefore, the barycenter map \eqref{center} would take values in $\R^k$, and $N(\Omega^+, \delta)$ would denote the corresponding neighborhood in $\R^k$. Those are the only modifications needed in order to adapt the above arguments to this general setting.

\end{remark}

\end{document}